\documentclass[11pt]{article}

\usepackage{times}
\usepackage{amsthm}
\usepackage{amsmath}
\usepackage{amssymb}
\usepackage{mathrsfs}

 \newtheorem{thm}{Theorem}[section]

 \newtheorem{prop}[thm]{Proposition}
 \theoremstyle{definition}
 \newtheorem{defn}[thm]{Definition}
 \theoremstyle{remark}
 \newtheorem{rem}[thm]{Remark}
 \newtheorem*{ex}{Example}
 \numberwithin{equation}{section}

\def\A{{\mathscr A}} \def\B{{\mathscr B}}  
  \def\H{\mathcal H} 
 
\def\S{\mathcal S} \def\G{\mathcal{G}}
 
\def\Si{\Sigma}

\def\I{{\rm 1\kern-.26em I}}

\def\1{\mathfrak{1}}

\def\M{\mathscr M}

\def\I{I\hspace{-1.7pt}I\hspace{-1pt}{\rm d}}

\def\({\left(}
\def\){\right)}
\def\[{\left[}
\def\]{\right]}
\def\<{\left<}
\def\>{\right>}

\begin{document}

%
%
%
%
%
%
%
%
%
\title{On Fr\'echet-Hilbert Algebras}

\date{\today}

\author{M. M\u antoiu and R. Purice \footnote{
\textbf{Primary  46L65, Secundary 35S05, 46K15.}
}
}
\date{\small}
\maketitle \vspace{-1cm}








\begin{abstract}
We consider Hilbert algebras with a supplementary Fr\'echet topology and get various extensions of the algebraic structure by using duality techniques.
In particular we obtain  optimal multiplier-type involutive algebras, which in applications are large enough to be of significant practical use. The setting covers many situations arising from quantization rules, as those involving square-integrable families of bounded operators
\end{abstract}

\bigskip

\bigskip
{\bf Address}

\medskip
Departamento de Matem\'aticas, Universidad de Chile,

Las Palmeras 3425, Casilla 653, Santiago, Chile

\emph{E-mail:} mantoiu@uchile.cl

\emph{E-mail:} Radu.Purice@imar.ro	

\bigskip
\medskip
{\bf Acknowledgements:}
The authors have been supported by the Chilean Science Foundation {\it Fondecyt} under the Grant 1120300.
RP acknowledges the partial support of a grant of the Romanian National Authority for Scientific Research, CNCS-UEFISCDI, project number
PN-II-ID-PCE-2011-3-0131 and the hospitality of the Universidad de Chile where part of this work has been done.

\maketitle

\section*{Introduction}\label{intro}

Hilbert algebras \cite{Di1} (see also Definition \ref{hilbalg}) play an important role in the theory of von Neumann algebras, group representations and Tomita-Takesaki theory. Each Hilbert algebra defines canonically two semifinite von Neumann algebras, each one being the commutant of the other. Reciprocally, a von Neumann algebra endowed with a faithful normal semi-finite trace defines a Hilbert algebra. In physics they often arise in the context of quantization theory, being non-commutative deformations of some classical commutative structure. Having such physical applications in view, we deal with the problem of restricting and (especially) extending the algebraical information contained in a Hilbert algebra to different spaces. Our aim is to perform this in a model-independent way, with a minimal extra structure involved.

In many concrete situations, besides the Hilbert norm, there is a second stronger Fr\'echet topology compatible with the involutive algebra structure.
We codify the entire structure in the first section of the paper under the name {\it Fr\'echet-Hilbert algebra}. In the second section we show
that in such a setting the algebraic structure extends considerably by duality techniques. In particular, one gets naturally
an optimal multiplier-type algebra that we call {\it the Moyal algebra}. Such an object has been studied in connection with the Weyl pseudodifferential calculus \cite{An1,GBV1,GBV2}, starting with the natural algebraico-topological structure of the Schwartz space $\S(\mathbb R^n)$. An adaptation for the Gelfand-Shilov spaces is contained in \cite{So}\,. Moyal algebras for the magnetic pseudodifferential theory \cite{IMP} were introduced and used in \cite{MP1}.

As said before, the abstractization we propose here works under somewhat minimal assumptions and does not need the setting of spaces of functions or distributions. It also opens the way to some developments and applications that will be the subject of a forthcoming article, which will also contain complements on the relevant topologies on the Moyal algebras, a representation theory by operators in locally convex spaces and applications to the matrix-valued pseudodifferential theory on compact Lie groups initiated in \cite{RT1,RT2,RTW}.

In \cite{BBM} a very general form of a symbolic calculus has been introduced and studied. It is defined by a family of bounded Hilbert space operators
indexed by a space $\Si$ and having a property of square-integrability with respect to a measure $\mu$ on $\Si$. The main novelty is that no topology or group properties are involved in the development of this calculus. Besides the usual pseudodifferential theory on $\mathbb R^n$
\cite{Fo,GBV1,GBV2}, many topics are particular cases of this approach, as is shown in \cite{BBM}. This includes twisted convolution algebras associated to projective group representations, the magnetic pseudodifferential calculus \cite{IMP,KO,MP1,MP4}, Weyl operators on nilpotent groups \cite{BB1,BB4,BB5,Pe}, pseudodifferential operators on Abelian locally compact groups \cite{GS,Ha,Se,We} and others. In a final section we are going to review briefly the constructions of \cite{BBM}, putting them in the perspective of Fr\'echet-Hilbert algebras. In this way the formalism of Moyal algebras will become available to the examples covered in \cite{BBM}, which is a novelty for some of them.

\section{Fr\'echet-Hilbert algebras}\label{mapsymb}

\begin{defn}\label{hilbalg}
A Hilbert algebra {\rm is a $^*$-algebra $(\A,\#,^\#)$ endowed with a scalar product $\<\cdot,\cdot\>:\A\times\A\rightarrow\mathbb C$ such that
\begin{enumerate}
\item
one has $\,\<g^\#,f^\#\>=\<f,g\>,\ \forall\,f,g\in\A$,
\item
one has $\,\<f\#g,h\>=\<g,f^\#\# h\>,\ \forall\,f,g,h\in\A$,
\item
for all $f\in\A$, the map ${\sf L}_f:\A\rightarrow\A,\,{\sf L}_f(g):=f\#g$ is continuous,
\item
$\A\#\A$ is total in $\A$.
\end{enumerate}
A complete Hilbert algebra is called an} $H^*$-algebra.
\end{defn}

Clearly one also has
\begin{equation*}\label{treica}
\<f\#g,h\>=\<f,h\# g^\#\>,\ \quad\forall\,f,g,h\in\A
\end{equation*}
and the map ${\sf R}_f:\mathscr A\rightarrow\A,\,{\sf R}_f(g):=g\#f$
is also continuous; therefore $\A\times\A\overset{\#}{\rightarrow}\A$ is separately continuous.

The completion $\mathscr B$ of $\A$ is a Hilbert space but in general it is no longer an algebra.
But the mappings ${\sf L}_f$ and ${\sf R}_f$ do extend to elements of $\mathbb B(\B)$, defining nondegenerate commuting representations ${\sf L,R}:\A\rightarrow\mathbb B(\B)$. 
We can regard ${\sf L}$ and ${\sf R}$ as giving separately continuous bilinear extensions
\begin{equation*}\label{alta}
\A\times\B\overset{\#}{\longrightarrow}\B,\quad{\rm and}\quad\B\times\A\overset{\#}{\longrightarrow}\B.
\end{equation*}

By taking weak closures in $\mathbb B(\mathscr B)$ one gets von Neumann algebras $\mathcal L(\A)$ and $\mathcal R(\A)$
which are the commutant of each other.

An element $f$ of $\mathscr B$ is called {\it bounded} if the mapping $\A\ni g\rightarrow f\# g\in\mathscr B$ is continuous
(or, equivalently, if the mapping $\A\ni g\rightarrow g\# f\in\mathscr B$ is continuous).
We denote by $\A^\flat\subset\mathscr B$ the space of all bounded elements;
it becomes naturally a Hilbert algebra containing $\A$ densely.
One defines \cite{Di1} on $\mathcal L(\A)$ (resp. $\mathcal R(\A)$)
{\it a normal faithful semifinite trace} $\tau_L$ (resp. $\tau_R$)
for whom the finite-trace operators correspond to the elements of ${\mathsf L}_{\A^\flat}$ (resp. ${\mathsf R}_{\A^\flat}$).

\begin{defn}\label{stricat}
A Fr\'echet $^*$-algebra {\rm is a $^*$-algebra $(\A,\#,^\#)$ with a Fr\'echet locally convex space topology
$\mathscr T$ such that the involution
$$
\A\ni f\rightarrow f^\#\in \A
$$
is continuous and the product
$$
\A\times\A\ni(f,g)\rightarrow f\# g\in\A
$$
is separately continuous.}
\end{defn}

\begin{rem}\label{mishto}
{\rm By Theorem 41.2 in \cite{Tr}, separate continuity of the map $\#$ implies hypocontinuity. This means that for any bounded subset $A$
of $\A$ the families of linear maps
$$
\{\A\ni g\mapsto f\# g\in\A\mid f\in A\}\quad{\rm and}\quad\{\A\ni g\mapsto g\# f\in\A\mid f\in A\}
$$
are equicontinuous.
}
\end{rem}

On the dual $\A^\dagger$ one can consider various topologies
$\mathcal T_\nu$ which are stronger than the weak$^*$-topology $\mathcal T_\sigma$ but weaker than the strong topology $\mathcal T_\beta$.
Such a topology $\mathcal T_\nu$ will be called {\it admissible} if it has as a basis of neighborhoods of the origin the polars
of a family $\mathcal B_\nu$ of bounded subsets of $\A$ satisfying the following:
\begin{enumerate}
\item
if $A,B\in\mathcal B_{\nu}$, there exists $C\in\mathcal B_\nu$ such that $A\cup B\subset C$,
\item
if $A\in\mathcal B_\nu$ and $\alpha\in\mathbb C$, then $\alpha A\subset B$ for some $B\in\mathcal B_\nu$,
\item
if $A\in\mathcal B_\nu$ and $f\in \A$ then $f\# A\in\mathcal B_\nu$ and $A\# f\in\mathcal B_\nu$.
\end{enumerate}

Let us write $\A^\dag_\nu$ for the dual $\A^\dag$ when considered with the
topology $\mathcal T_\nu$. A net $\{F_\lambda\}_{\lambda\in\Lambda}$ converges to $0$ in $\A_\nu^\dag$ if and only if $F_\lambda(g)$ converges to $0$ uniformly in $g\in A$ for every $A\in\mathcal B_\nu$. Both the weak$^*$-topology $\mathcal T_\sigma$ and the strong topology
$\mathcal T_\beta$ are admissible topologies; for this one takes $\mathcal B_{\sigma}$ the family of all finite subsets of $\A$
and $\mathcal B_\beta$ the family of all bounded subsets of $\A$. Another interesting example is the family $\mathcal B_\gamma$
of all compact subsets of $\A$, leading to the topology $\mathcal T_\gamma$ of convergence which is uniform on compact sets.
By \cite[Th. 33.1]{Tr}, the bounded subsets of $\A^\dag_\nu$ are the same for all admissible topologies $\mathcal T_\nu$ (and are exactly the equicontinuous subsets).

\begin{rem}\label{pripri}
{\rm Weak duals of infinite-dimensional Fr\'echet spaces are definitely not complete. But it is known that they are quasi-complete.
On the other hand, the strong dual $\A^\dag_\beta$ (the dual $\A^\dag$ endowed with the strong topology $\mathcal T_\beta$ described above) of any Fr\'echet space is complete.
}
\end{rem}

We introduce now our main mathematical object.

\begin{defn}\label{tot}
A Fr\'echet-Hilbert algebra $(\A,\#,^\#,\mathscr T,\<\cdot,\cdot\>)$ {\rm is both a Fr\'echet $^*$-algebra and a Hilbert algebra,
the topology $\mathscr T$ being assumed finer that the norm topology associated to the scalar product.}
\end{defn}

As before, we denote by $\B$ the Hilbert space completion of $\A$; by the Riesz Lemma it is identified with its strong dual $\B^\dagger$.
We also denote by $\iota$ the canonical inclusion of $(\A,\mathscr T)$ into $(\B,\parallel\cdot\parallel)$; it is continuous and has dense range.
For each admissible topology $\mathcal T_\nu$ on the dual, using Riesz' identification, one gets a continuous linear injection
$\iota^\dagger:(\mathscr B,\parallel\cdot\parallel)\rightarrow(\A^\dagger,\mathcal T_\nu)$;
thus any Fr\'echet-Hilbert algebra $\A$ generates {\it a Gelfand triple} $(\A,\B,\A^\dagger_\nu)$.
Actually we are going to treat $(\A,\B,\A^\dagger)$ as the Gelfand triple and specify the topology $\mathcal T_\nu$ when needed. The duality between $\A$ and $\A^\dag$ will be denoted by $\<\cdot,\cdot\>$ because it is consistent with the scalar product of $\B$.

\begin{rem}\label{pripri}
{\rm Note that $\B$ (and hence $\A$) is dense in $\A^\dag_\sigma$ \cite[Prop. 35.4]{Tr}. 
}
\end{rem}

\section{Extensions of the product and Moyal algebras}\label{injurac}

Let us fix a Fr\'echet-Hilbert algebra $(\A,\#,^\#,\mathscr T,\<\cdot,\cdot\>)$ and an admissible topology $\mathcal T_\nu$ on the topological dual $\mathscr A^\dagger$, given by a family $\mathcal B_\nu$ of bounded subsets of $\A$ as above.

\begin{prop}\label{fosforos}
\begin{enumerate}
\item
The composition law $\#$ extends to bilinear separately continuous mappings
$\#:\A\times\A^\dagger_\nu\rightarrow\A^\dagger_\nu$ and $\#:\A^\dagger_\nu\times\A\rightarrow\A^\dagger_\nu$.
\item
For any $f\in\A^\dagger$ and $g,h\in\A$ one has
\begin{equation*}\label{urce}
f\#(g\# h)=(f\# g)\#h,
\end{equation*}
\begin{equation*}\label{urcu}
h\#(g\# f)=(h\# g)\#f,
\end{equation*}
\begin{equation*}\label{urca}
(g\# f)\# h=g\#(f\# h).
\end{equation*}
\item
The involution $^\#$ extends to a topological anti-linear isomorphism $^\#:\A_\nu^\dag\rightarrow\A_\nu^\dag$ which is an involution,
such that for every $f\in\A^\dag$ and $g\in\A$ one has
\begin{equation*}\label{stuchid}
(f\# g)^\#=g^\#\# f^\#,\quad\ (g\# f)^\#=f^\#\# g^\#.
\end{equation*}
\end{enumerate}
\end{prop}

\begin{proof}
1. We are going to justify only the second extension; the first one follows in the same way. For $f\in\A^\dag$ and $g,h\in\A$ one sets
\begin{equation*}\label{diul}
\<f\# g,h\>:=\<f,h\# g^\#\>.
\end{equation*}
Clearly this defines an element $f\# g$ of the topological dual of $\A$, which coincides with the one given by the Hilbert algebra structure of $\A$ if $f\in\A$.

We still have to show separate continuity. 

First fix $g\in\A$ and pick a net $\{f_\lambda\}_{\lambda\in\Lambda}$ converging to $0$ in $\A^\dag_\nu$.
This means that for every $B\in\mathcal B_\nu$ one has $\<f_\lambda,k\>\rightarrow 0$ uniformly in $k\in B$. 
For any $A\in\mathcal B_\nu$, since $B:=A\#g^\#\in\mathcal B_\nu$, one has 
$$
\<f_\lambda\# g,h\>=\<f_\lambda,h\# g^\#\>\rightarrow 0\ \,{\rm uniformly\ in}\ h\in A,
$$
thus $f_\lambda\# g\rightarrow 0$ in $\A^\dag_\nu$.

Now fix $f\in\A^\dag$ and assume that $g_\lambda\rightarrow 0$ in $\A$.
For any $A\in\mathcal B_\nu$ it follows from Remark \ref{mishto} that $h\# g_\lambda^\#\rightarrow 0$ uniformly in $h\in A$. Then
$$
\< f\# g_\lambda,h\>=\<f,h\# g^\#_\lambda\>\rightarrow 0\ \,{\rm uniformly\ in}\ h\in A,
$$
thus $f\# g_\lambda\rightarrow 0$ in $\A^\dag_\nu$ and we are done.

\medskip
2. Recall that $\A$ is dense in $\A^\dag_\sigma$. Then the associativity properties  follow easily
 by approximation from the associativity of the composition law $\#$ in $\A$ and from the continuity property we have just proved, with $\nu=\sigma$.

\medskip
3. The extension of $^\#$ is defined by transposition: if $f\in\A_\nu^\dag$ one sets
$$
\<f^\#,h\>:=\overline{\<f,h^\#\>},\quad\ \forall\,h\in\A.
$$
It follows immediately that it is an involution. It is a topological isomorphism for any $\nu$ because the family $\mathcal B_\nu$ is stable under the involution of $\mathscr A$. For $f\in\A^\dag$ and $g,h\in\A$ we compute using the definition of the involution and the axioms of a Hilbert algebra
$$
\begin{aligned}
\<(f\# g)^\#,h\>&=\overline{\<f\# g,h^\#\>}=\overline{\<f,h^\#\# g^\#\>}=\overline{\<(h^\#\# g^\#)^\#,f^\#\>}\\
&=\overline{\<g\# h,f^\#\>}=\<f^\#,g\#h\>=\<g^\#\# f^\#,h\>,
\end{aligned}
$$
so the first equality is proven. The second follows similarly.
\end{proof}

\begin{defn}
{\rm Let $\mathscr A$ be a Fr\'echet-Hilbert algebra with topological dual $\A^\dagger$\,. One introduces}
\begin{enumerate}
\item {\it the right Moyal algebra}
$\,\M_R:=\left\{f\in\A^\dagger\,\mid\, \A\# f\subset \A\right\}$,

\item {\it The left Moyal algebra} $\,\M_L:=\left\{f\in\A^\dagger\,\mid\, f\#\,\A\subset \A\right\}$,
\item {\it The (bi-sided) Moyal algebra}

$\,\M:=\M_R\cap \M_L=\left\{f\in\A^\dagger\,\mid\, f\#\,\A\subset \A\supset\A\# f\right\}$.
\end{enumerate}
\end{defn}

\noindent
Thus we have bilinear maps
\begin{equation*}\label{oaial}
\A\times\M_R\overset{\#}{\longrightarrow}\A,\quad\ \M_L\times\A\overset{\#}{\longrightarrow}\A.
\end{equation*}
If $f\in\A,\,g\in\M_L$ and $h\in\M_R$, by using Proposition \ref{fosforos} and the definitions we check immediatly that the identity $(g\# f)\# h=g\#(f\# h)$ holds in $\A$. One also shows easily that $\M_R^\#=\M_L$ and $\M_L^\#=\M_R$, implying that
$\M^\#=\M$.

\begin{prop}\label{ptak}
One gets bilinear extensions
\begin{equation*}\label{bil}
\M_R\times\A^\dag\overset{\#}{\longrightarrow}\A^\dag,\ \quad\A^\dag\times\M_L\overset{\#}{\longrightarrow}\A^\dag.
\end{equation*}
If $f\in\A^\dag,\,g\in\M_R$ and $h\in\M_L$ we have
$$
(g\# f)\# h=g\#(f\# h).
$$
\end{prop}

\begin{proof}
We treat the first one. For $f\in\M_R$ and $g\in\A^\dag$ define $f\# g\in\A^\dag$ by
\begin{equation*}\label{incerc}
\<f\# g,h\>:=\<g,f^\#\# h\>,\quad\ \forall\,h\in\A.
\end{equation*}
The right-hand side is clearly linear in $h$, but we also need to justify continuity.
By a simple application of the Closed Graph Theorem, if $k\in\M_L$ then $\A\ni h\mapsto k\# h\in\A$ is continuous.
Taking $k=f^\#$ and recalling that $g$ is a continuous functional, we conclude that $f\# g\in\A^\dag$ indeed.

Associativity follows if we apply the definitions and the associativity properties already obtained.
\end{proof}

We gather the basic properties of $\M_L,\M_R$ and $\M$ in the following statement.

\begin{prop}\label{complecar}
\begin{enumerate}
\item
Both $\M_L$ and $\M_R$ are algebras.
\item
The involution $^\#$ on $\A^\dagger$ restricts to reciprocal anti-linear isomorphisms $^\#:\M_L\rightarrow\M_R$ and $\,^\#:\M_R\rightarrow\M_L$ satisfying
$$
(f\# g)^\#=g^\#\# f^\#,\ \ \ \ \ \forall f,g\in \M_L \ (or\ \forall f,g\in \M_R).
$$
\item
$\M$ is an involutive algebra.
\item
One has $\A\#\M_R\subset\A$ and $\M_L\#\A\subset\A$\,, thus $\A$ a self-adjoint bi-sided ideal in the $^*$-algebra $\M$.
\end{enumerate}
\end{prop}

\begin{proof}
1. We treat $\M_L$. Let $f,g\in\M_L$ and $h\in\A$; by Proposition \ref{ptak} we need to show that $(f\# g)\# h\in\A$. This follows from the definition of
$\M_L$ if the equality $(f\# g)\# h=f\#(g\# h)$ is established. We are going to get it working weakly on any $k\in\A$:
$$
\begin{aligned}
\<(f\# g)\# h,k\>&=\<f\# g,k\#h^\#\>=\<f,(k\# h^\#)\# g^\#\>\\
&=\<f,k\#(h^\#\# g^\#)\>=\<f,k\#(g\# h)^\#\>\\
&=\<f\#(g\# h),k\>.
\end{aligned}
$$
During the computation we used an associativity relation that is already known.

There is still some associativity to prove, but we leave it to the reader.

\medskip
2. This has been mentioned above and is easy.

\medskip
3 follows from 1 and 2, since $\M$ is the intersection of $\M_L$ and $\M_R$.

\medskip
4 is obvious.
\end{proof}

\begin{rem}\label{prip}
{\rm By inspection it can be shown that the locally convex space $(\A,\mathscr T)$ only needs to be barrelled and $B_r$-convex \cite{Sch}, the main issue being the validity of the Closed Graph Theorem. This extra generality, needed to cover Gelfand-Shilov spaces in connection with the Weyl calculus \cite{So}, will not be considered here
}
\end{rem}

\begin{rem}\label{pripi}
{\rm In \cite{GBV1,MP1,So}, studying the particular case of the (magnetic) Weyl composition on Schwartz or Gelfand Shilov spaces, it is shown that the Moyal algebras are much larger than the initial $\A$. In the Weyl quantization, many symbols from the Moyal algebra are turned into unbounded operators on $L^2$.
}
\end{rem}

\begin{rem}\label{pipri}
{\rm However the Hilbert space $\B$ and the Moyal algebras are not comparable in general. For the Weyl pseudodifferential calculus with $\A=\S(\mathbb R^{2n})$ one has $\B=L^2(\mathbb R^{2n})$ and $\A^\dagger$ is the space of tempered distributions. In the Schr\"odinger representation, the Weyl correspondence maps isomorphically $\B$ to the ideal of Hilbert-Schmidt operators in $L^2(\mathbb R^n)$, while $\M_L$
corresponds to linear continuous operators in $\S(\mathbb R^n)$ and $\M_R$ corresponds to linear continuous operators in $\S'(\mathbb R^n)$.
Then clearly no inclusion is available (think of rank-one operators $|u\rangle\langle v|$ with various types of vectors $u,v$).
}
\end{rem}

\begin{ex}\label{rice}
{\rm Many examples of Fr\'echet-Hilbert algebras are constructed from the Gelfand triple
$\left(\A=\mathcal S(\mathbb R^m),\,\mathscr B=L^2(\mathbb R^m),\,\A^\dag=\mathcal S'(\mathbb R^m)\right)$,
where $\mathcal S(\mathbb R^m)$ is the Schwartz space endowed both with the $L^2$ scalar product and with its standard Fr\'echet topology
and $\mathcal S'(\mathbb R^m)$ is the space of all tempered distributions on $\mathbb R^m$. The simplest non-trivial situation
is $f\# g=fg$ (pointwise product) and $f^\#=\overline{f}$ (complex conjugation). In this case one has
$\A^\flat=L^2(\mathbb R^m)\cap L^\infty(\mathbb R^m)$, $\mathcal L(\A)=\mathcal R(\A)\cong L^\infty(\mathbb R^m)$
and $\mathscr M_L=\mathscr M_R=\mathscr M=C^\infty_{{\rm pol}}(\mathbb R^m)$
($C^\infty$ functions with polynomially bounded derivatives).
One sees from this example that although $\mathscr M$ can be quite big (containing all the H\"ormander classes of symbols \cite{Fo}), it does not contain the Hilbert space $\mathscr B$ or at least
$\A^\flat$. It is also clear that $\A\#\A^\dag$ is not contained in any of the Moyal algebras.}
\end{ex}

\section{Fr\'echet-Hilbert algebras associated to square-integrable families of bounded operators}\label{endix}

We are given a measure space $(\Si,\mu)$ (as a locally compact topological space endowed with a Radon measure, for instance). Let $\{\pi(s)\!\mid\! s\in\Si\}\subset\mathbb B(\H)$ be a family of bounded operators in the separable complex Hilbert space $\H$. We assume that $s\to\pi(s)$ is weakly measurable and satisfies the condition
$$
\int_\Si\! d\mu(s)\,|\<\pi(s)u,v\>_\H|^2=\,\parallel\!u\!\parallel^2\,\parallel\!v\!\parallel^2,\quad\forall\,u,v\in\H.
$$
The space $\Si$ is not a group and we don't know anything on the products $\pi(s)\pi(t)$. We are going to describe some constructions from \cite{BBM}, to which we send for proofs, technical details and further information.

Let us set $\Phi_{u\otimes v}(\cdot):=\<\pi(\cdot)u,v\>$; this defines an isomorphism from $\H\otimes\overline\H$ to a closed subspace $\mathscr B^2(\Si)$ of $L^2(\Si)$ (we set $\overline\H$ for the opposite of the Hilbert space $\H$ and $\otimes$ for the Hilbert space tensor product). In many (but not all) particular cases one has $\mathscr B^2(\Si)=L^2(\Si)$. Then a correspondence $\Pi:\mathscr B^2(\Si)\rightarrow \mathbb B(\H)$ is defined essentially by
$$
\Pi(f):=\int_\Si\!d\mu(s)f(s)\pi(s)^*.
$$
This correspondence sends isomorphically $\mathscr B^2(\Si)$ in $\mathbb B_2(\H)$ (the ideal of all Hilbert-Schmidt operators on $\H$).
Actually one has
$$
\<\Pi(f),\Pi(g)\>_{\mathbb B^2(\H)}:={\rm Tr}[\Pi(f)\Pi(g)^*]=\int_\Si\!d\mu(s)f(s)\overline{g(s)}=:\<f,g\>_{L^2(\Si)}.
$$
The mapping $\Pi$ is uniquely defined by the relation $\Pi(\Phi_{u\otimes v})=\<\cdot,v\>u$, specifying the way we obtain the rank one operators.
Being a closed subspace of $L^2(\Si)$, the set of "symbols" $\mathscr B^2(\Si)$ is a Hilbert space. We complete this structure to a $H^*$-algebra (Def. \ref{hilbalg}) $\(\mathscr B^2(\Si),\#,^\#,\<\cdot,\cdot\>_{L^2(\Si)}\)$
by transporting the $^*$-algebra structure of $\mathbb B(\H)$. So the product is defined by
$f\#g:=\Pi^{-1}[\Pi(f)\Pi(g)]$ and the involution by $f^\#:=\Pi^{-1}[\Pi(f)^*]$.
More explicit expressions are deduced in \cite{BBM} but they are not needed here.

Suppose now given a Fr\'echet space $\G$ continuously and densely embedded in the Hilbert space $\H$. This is very common in applications, where $\H$ might be a $L^2$-space, $\G$ could be a space of more regular functions, with a natural Fr\'echet topology, and the topological dual $\G^\dagger$ is some space of distributions. No compatibility between $\G$ and the basic family $\{\pi(s)\!\mid\! s\in\Si\}$ is needed. Let us denote by $\G\widehat\otimes\overline\G$ the completed projective tensor product and define $\mathscr G(\Si):=\Phi(\G\widehat\otimes\overline\G)$ with the  topology transported by $\Phi$ from $\G\widehat\otimes\overline\G$. It is shown in \cite{BBM} that $\mathscr G(\Si)$ is a Fr\'echet-Hilbert algebra; thus we can take $\A=\mathscr G(\Si)$ and $\mathscr B=\mathscr B^2(\Si)$ in the preceding sections.

\begin{rem}\label{puzau}
{\rm Actually the Fr\'echet-Hilbert algebra described above admits a representation in the Gelfand triple $(\mathcal G,\mathcal H,\mathcal G^\dag)$. By definition, such a representation is an isomorphism $\Pi:\A^\dag\rightarrow\mathbb B(\mathcal G,\mathcal G^\dag)$ which restricts to an isomorphism $\Pi:\A\rightarrow\mathbb B(\mathcal G^\dag,\mathcal G)$ and satisfies
$\Pi(f\# g)=\Pi(f)\Pi(g)$ and $\Pi(f)^*=\Pi(f^\#)$ for all $f,g\in\A$.
In such a setting one shows easily that $\A\#\A^\dag\#\A\subset\A$, which implies immediately that $\A\#\A^\dag\subset\mathscr M_L$ and $\A^\dag\#\A\subset\mathscr M_R$. As noticed in Example \ref{rice}, this does not hold for all Fr\'echet-Hilbert algebras.
}
\end{rem}

Concrete examples are given in \cite{BBM}, along the lines described in the Introduction; it can be seen that, although $\mathscr G(\Si)$ is defined indirectly, in applications it coincides with some known useful Fr\'echet space. However its $^*$-algebra structure can be highly non-trivial. Our approach makes available the Moyal algebra formalism for all these situations.



\begin{thebibliography}{00}

\bibitem{An1} M. A. Antonets, \textit {The Classical Limit for Weyl Quantization}, Lett. Math. Phys. \textbf{ 2}, (1978), 241--245.

\bibitem{BB1} I. Beltit\u a and D. Beltit\u a, \textit {Magnetic Pseudo-differential Weyl Calculus on Nilpotent Lie Groups.}
Ann. Global Anal. Geom. \textbf{ 36}, (3), (2009), 293--322.

\bibitem{BB4} D. Beltit\u a and I. Beltit\u a, \textit {A Survey on Weyl Calculus for Representations of Nilpotent Lie Groups},
In: S.T. Ali, P. Kielanowski, A. Odzijewicz, M. Schlichenmeier, Th. Voronov (eds.), XXVIII Workshop on Geometric Methods in Physics, (2009), 7--20.

\bibitem{BB5} I. Belti\c t\u a and D. Belti\c t\u a, \textit {Continuity of Magnetic Weyl Calculus}, J. Funct. Analysis, \textbf {260} (7), (2011), 1944--1968.

\bibitem{BBM} I. Belti\c t\u a, D. Belti\c t\u a and M. M\u antoiu, \textit {Quantization and Dequantization via Square-Integrable Families of Operators}, ArXiV and submitted.

\bibitem{Di1} J. Dixmier: \textit { Von Neumann Algebras}, North Holland, Amsterdam, New York, Oxford, (1981).


\bibitem{Fo} G. B. Folland: \textit { Harmonic Analysis in Phase Space}, Princeton Univ. Press, Princeton, NJ, (1989).

\bibitem{GBV1} J. M. Gracia Bondia and J. C. Varilly, \textit { Algebras of Distributions Suited to Phase-Space Quantum Mechanics. I}, J. Math. Phys, \textbf{  29}, (1988), 869--878.

\bibitem{GBV2} J. M. Gracia Bondia and J. C. Varilly, \textit { Algebra of Distributions Suitable for Phase Space Quantum Mechanics. II,
Topologies on the Moyal algebra}, J. Math. Phys. \textbf{ 29}, (1988), 880--887.

\bibitem{GS} K. Gr\"ochenig, T. Strohmer, \textit {Pseudodifferential Operators on Locally Compact Abelian Groups and Sj\"ostrand's Symbol Class}.
J. reine angew. Math. \textbf{  613}, (2007), 121--146.

\bibitem{Ha} S.~Haran: \textit { Quantizations and Symbolic Calculus over the $p$-adic Numbers}. Ann. Inst. Fourier (Grenoble) \textbf{  43} (4), (1993), 997--1053.

\bibitem{IMP} V. Iftimie, M. M\u antoiu and R. Purice, \textit { Magnetic Pseudodifferential Operators}, Publ. RIMS. \textbf{  43}, (2007), 585--623.


\bibitem{KO} M. V. Karasev, T. A. Osborn, \textit { Symplectic Areas, Quantization, and Dynamics in Electromagnetic Fields}.
J. Math. Phys. \textbf{  43} (2), (2002), 756--788.


\bibitem{MP1} M. M\u antoiu and R. Purice, \textit {The Magnetic Weyl Calculus}, J. Math. Phys. \textbf{  45} (4), (2004),  1394--1417.

\bibitem{MP4} M. M\u antoiu, R. Purice:  \textit { The Modulation Mapping for Magnetic Symbols and Operators}. Proc. Amer. Math. Soc. \textbf{  138} (8), (2010), 2839--2852.

\bibitem{Pe} N. V. Pedersen, \textit { Matrix Coefficients and a Weyl Correspondence for Nilpotent Lie Groups}, Invent. Math. \textbf{ 118}, (1994), 1--36.

\bibitem{RT1} M. Ruzhanski and V. Turunen, \textit { Pseudo-differential Operators and Symmetries: Background Analysis and Advanced Topics}, Birkh\" auser, Basel, 2010. 724pp.

\bibitem{RT2} M. Ruzhanski and V. Turunen, \textit { Global Quantization of Pseudo-differential Operators on Compact Lie Groups, SU(2) and 3-Sphere}, Int. Math. Res. Notices \textbf{ 2013}, (2013), 2439--2496.

\bibitem{RTW} M. Ruzhansky, V. Turunen and J. Wirth, \textit { H\" ormander Class of Pseudo-differential Operators on Compact Lie Groups and Global Hypoellipticity}, J. Fourier Anal. Appl., {\bf 20} (2014), 476--499.

\bibitem{Sch} H. Schaefer: \textit { Topolgical Vector Spaces}, Springer-Verlag, New York, Heidelberg, Berlin, 1971.

\bibitem{Se} I. E. Segal: \textit {Transforms for Operators and Symplectic Automorphisms over a Locally Compact Abelian Group}.
Math. Scand. \textbf{ 13}, (1963), 31--43.

\bibitem{So} M. A. Soloviev, \textit { Moyal Multiplier Algebras of the Test Function Spaces of Type S}, J. Math. Phys. \textbf{ 52} (6), p063502 (2011).

\bibitem{Tr} F. Tr\`eves, \textit { Topological Vector Spaces, Distributions and Kernels}, Academic Press, 1967.

\bibitem{We} A. Weil: \textit { Sur certains groupes d'op\'erateurs unitaires}. Acta Math. \textbf{ 111}, (1964), 143--211.


\end{thebibliography}
\end{document}